\documentclass[12pt]{amsart}

\usepackage{color}
\usepackage{amsmath, amsthm, amssymb}
\usepackage{amsfonts}
\usepackage[ansinew]{inputenc}
\usepackage[dvips]{epsfig}
\usepackage{graphicx}
\usepackage[english]{babel}
\usepackage{hyperref}
\theoremstyle{plain}
\newtheorem{thm}{Theorem}[section]
\newtheorem{cor}[thm]{Corollary}
\newtheorem{lem}[thm]{Lemma}
\newtheorem{prop}[thm]{Proposition}

\theoremstyle{definition}

\theoremstyle{remark}
\newtheorem{rem}[thm]{Remark}

\numberwithin{equation}{section}

\newcommand{\de}{\partial}

\newcommand{\fls}{(-\Delta)^s}
\newcommand{\R}{\mathbb{R}}

\newcommand{\N}{\mathbb{N}}

\newcommand{\average}{{\mathchoice {\kern1ex\vcenter{\hrule height.4pt
width 6pt depth0pt} \kern-9.7pt} {\kern1ex\vcenter{\hrule
height.4pt width 4.3pt depth0pt} \kern-7pt} {} {} }}

\def\R{\mathbb{R}}

\begin{document}

\title[Boundary regularity for the fractional heat equation]{Boundary regularity for \\the fractional heat equation}

\author{Xavier Fern\'andez-Real}

\address{Universitat Polit\`ecnica de Catalunya, Facultat de Matem\`atiques i Estad\'istica, Pau Gargallo 5, 08028 Barcelona, Spain}

\email{xavier.fernandez-real@estudiant.upc.edu}

\author{Xavier Ros-Oton}

\address{University of Texas at Austin, Department of Mathematics, 2515 Speedway, TX 78712 Austin, USA}

\email{ros.oton@math.utexas.edu}

\keywords{Fractional Laplacian, fractional heat equation, boundary regularity.\\ $\ast$ \hspace{0.5mm} This work is part of the bachelor's degree thesis of the first author.}

\maketitle

\begin{abstract}
We study the regularity up to the boundary of solutions to fractional heat equation in bounded $C^{1,1}$ domains.

More precisely, we consider solutions to $\de_t u + \fls u=0 \textrm{ in }\Omega,\ t > 0$, with zero Dirichlet conditions in $\R^n\setminus \Omega$ and with initial data $u_0\in L^2(\Omega)$.
Using the results of the second author and Serra for the elliptic problem, we show that for all $t>0$ we have $u(\cdot, t)\in C^s(\R^n)$ and $u(\cdot, t)/\delta^s \in C^{s-\epsilon}(\overline\Omega)$ for any $\epsilon > 0$ and $\delta(x) = \textrm{dist}(x,\de\Omega)$.

Our regularity results apply not only to the fractional Laplacian but also to more general integro-differential operators, namely those corresponding to stable L\'evy processes.

As a consequence of our results, we show that solutions to the fractional heat equation satisfy a Pohozaev-type identity for positive times.
\end{abstract}

\section{Introduction and results}

The aim of this paper is to study the regularity of solutions to the fractional heat equation
\begin{equation}\label{eq.par}
  \left\{ \begin{array}{rcll}
  \de_t u + \fls u&=&0& \textrm{in }\Omega,\ t > 0 \\
  u&=&0& \textrm{in } \R^n\setminus \Omega,\ t > 0 \\
  u(x,0)&=&u_0(x)&\textrm{in }\Omega,\ \textrm{for } t = 0,
  \end{array}\right.
\end{equation}
and also to more general nonlocal parabolic equations $\partial_t u+Lu=0$ in bounded $C^{1,1}$ domains $\Omega$.
Here, $\fls$ is the fractional Laplacian, defined by
\[
(-\Delta)^s u (x)= c_{n,s}{\rm PV}\int_{\R^n}\frac{u(x)-u(y)}{|x-y|^{n+2s}}dy
\]
for $c_{n,s}$ a normalization constant.

The regularity of solutions to nonlocal equations is one of the hot topics in nonlinear analysis nowadays.
In particular, the regularity of solutions to parabolic problems like \eqref{eq.par} or related ones have been studied in \cite{CKS, BGR, CCV, FK, KS, CV, CSV, dPQR, CD, CD2, Serra, CVasseur, CF}.
Still, most of these works deal with the interior regularity of solutions, and there are few results concerning the regularity up to the boundary.
This is the topic of study in this paper.

For the elliptic problem
\begin{equation}
  \label{eq.ell}
  \left\{ \begin{array}{rcll}
  \fls u&=&g(x)& \textrm{in }\Omega \\
  u&=&0& \textrm{in } \R^n\setminus \Omega,
  \end{array}\right.
\end{equation}
the boundary regularity of solutions is quite well understood, as explained next.

When $\Omega=B_1$ and $g\equiv1$, the solution of problem \eqref{eq.ell} has a simple explicit expression.
Indeed, this solution is given by $u(x) = c(1-|x|^2)^s$ in $B_1$, for some constant $c>0$.
Notice that this solution is not smooth up to the boundary, but it is only $C^s(\overline\Omega)$.

It turns out that this boundary behavior is the same for all solutions, in the sense that any solution to \eqref{eq.ell} satisfies
\begin{equation}\label{eq.bounds}
-C\delta^s \leq u \leq C \delta^s\textrm{ in } \Omega,
\end{equation}
where $\delta(x)=\textrm{dist}(x,\partial\Omega)$.
Combining this bound with the well known interior regularity estimates, one finds that $u\in C^s(\overline\Omega)$ (see Proposition 1.1 in \cite{RS}).

Furthermore, for any $g\in L^\infty(\Omega)$, the solution $u$ to the elliptic problem \eqref{eq.ell} satisfies that $u/\delta^s$ is H\"older continuous up to the boundary, and
\begin{equation}\label{quotient}
\|u/\delta^s \|_{C^\alpha(\overline\Omega)} \leq C \|g\|_{L^\infty(\Omega)},
\end{equation}
for some $\alpha>0$ small.
This is the main result of \cite{RS}.

For the fractional heat equation \eqref{eq.par}, Chen-Kim-Song \cite{CKS} established sharp two-sided estimates for the heat kernel of \eqref{eq.par} in $C^{1,1}$ domains (see also \cite{BGR}).
Their estimates yield in particular a bound of the form \eqref{eq.bounds} for solutions to \eqref{eq.par} for positive times, and this implies that solutions $u(x,t)$ satisfy $u(\cdot,t)\in C^s(\overline\Omega)$ for $t>0$.
Nevertheless, the results in \cite{CKS,BGR} do not yield any estimate like \eqref{quotient}.

The aim of the present paper is to show that, for $t>0$, solutions to \eqref{eq.par} satisfy \eqref{quotient}.
Namely, we will show that, for any fixed $t_0>0$, solutions $u(x,t)$ satisfy
\[
\sup_{t \geq t_0} \left\|\frac{u(\cdot, t)}{\delta^s}\right\|_{C^\alpha(\overline{\Omega})} \leq C(t_0)\|u_0\|_{L^2(\Omega)}.
\]
To prove this, we use the results of the second author and Serra \cite{RS}.

\subsection{Main result}

As explained next, our main boundary regularity result apply not only to the fractional Laplacian, but also to more general integro-differential equations.

Indeed, we consider infinitessimal generators of stable and symmetric L\'evy processes.
These operators are uniquely determined by a finite measure on the unit sphere $S^{n-1}$, often referred as the {spectral measure} of the process.
When this measure is absolutely continuous, symmetric stable processes have generators of the form
\begin{equation} \label{eq.gen1}
 Lu(x) = \int_{\R^n} \bigl(2u(x)-u(x+y)-u(x-y)\bigr)\frac{a(y/|y|)}{|y|^{n+2s}}\,dy,
\end{equation}
where $s \in (0,1)$ and $a$ is any nonnegative function in $L^1(S^{n-1})$ satisfying $a(\theta) = a(-\theta)$ for $\theta\in S^{n-1}$.

The most simple and important case of stable process is the fractional Laplacian, whose spectral measure is constant on $S^{n-1}$ (so that the process is isotropic).

The regularity theory for general operators of the form \eqref{eq.gen1} has been recently developed by the second author and Serra in \cite{RS-F}; see also \cite{RS-L}.
Using the boundary regularity results established in \cite{RS-F}, we will show the following.

This is the main result of the present paper.

\begin{thm} \label{thm.main}
Let $\Omega\subset\R^n$ be any bounded $C^{1,1}$ domain, and $s\in(0,1)$.
Let $L$ be any operator of the form \eqref{eq.gen1} and satisfying the ellipticity conditions
\begin{equation}\label{eq.gen2}
 \Lambda_1 \leq \inf_{\nu\in S^{n-1}}\int_{S^{n-1}}|\nu\cdot\theta|^{2s}a(\theta)d\theta,  \qquad 0\leq a(\theta)\leq \Lambda_2 \quad \textrm{for all}\ \theta\in S^{n-1},
\end{equation}
where $\Lambda_1$ and $\Lambda_2$ are positive constants.

Let $u_0\in L^2(\Omega)$, and let $u$ be the solution to
\begin{equation} \label{eq.par-L}
  \left\{ \begin{array}{rcll}
  \de_t u + L u&=&0& \textrm{in }\Omega,\ t > 0 \\
  u&=&0& \textrm{in } \R^n\setminus \Omega,\ t>0 \\
  u(x,0)&=&u_0(x)&\textrm{in }\Omega,\ \textrm{for } t = 0.
  \end{array}\right.
\end{equation}
Then,
\begin{itemize}
 \item[(a)] For each $t_0 > 0$,
 \begin{equation}
 \label{eq.bound.a}
  \sup_{t \geq t_0} \|u(\cdot, t)\|_{C^s(\R^n)} \leq C_1(t_0)\|u_0\|_{L^2(\Omega)}.
 \end{equation}
 \item[(b)] For each $t_0 > 0$,
 \begin{equation}
 \label{eq.bound.b}
  \sup_{t \geq t_0} \left\|\frac{u(\cdot, t)}{\delta^s}\right\|_{C^{s-\epsilon}(\overline{\Omega})} \leq C_2(t_0)\|u_0\|_{L^2(\Omega)},
 \end{equation}
 where $\delta(x) = \textrm{dist}(x, \de \Omega)$, and for any $\epsilon > 0$.
\end{itemize}
The constants $C_1$ and $C_2$ depend only on $t_0$, $n$, $s$, $\epsilon$, $\Omega$, and the ellipticity constants $\Lambda_1$ and $\Lambda_2$.
Moreover, these constants $C_1$ and $C_2$ blow up as $t_0 \downarrow 0$.
\end{thm}

The main result is the second part of the Theorem ---the H\"older regularity of $u/d^s$ up to the boundary.
As explained above, this is new even for the fractional Laplacian.

Notice that our result applies to spectral measures that may vanish in a subset of $S^{n-1}$.
In fact, the spectral measure could even be of the form $a=\chi_\Sigma$, where $\Sigma\subset S^{n-1}$ is any subset with $|\Sigma|>0$ and such that $\Sigma=-\Sigma$.

The main idea of the proof is to write the solution of the heat equation in terms of the eigenfunctions, check that these eigenfunctions fulfil bounds of the form \eqref{eq.bound.a}-\eqref{eq.bound.b}, and deduce the desired result.
Still, the proof is not immediate, since one has to prove first that the eigenfunctions belong to $L^\infty$, with an explicit bound of the form $\|\phi_k\|_{L^\infty}\leq C k^q$ for some fixed exponent $q$, and being $\phi_k$ the $k$-th eigenfunction.

To do this, we need to establish the following result, which may be of independent interest.

\begin{prop}
 \label{prop.maingen}
 Let $\Omega\subset \R^n$ be any bounded domain, $s\in(0,1)$, $g\in L^2(\Omega)$, and $u$ be the weak solution of
  \begin{equation}
  \label{eq.propmaingen}
    \left\{ \begin{array}{rcll}
    L u&=&g& \textrm{in }\Omega\\
    u&=&0& \textrm{in } \R^n\setminus \Omega,
    \end{array}\right.
  \end{equation}
 where $L$ is a nonlocal operator of the form \eqref{eq.gen1} and \eqref{eq.gen2}. 
 Assume in addition that $g\in L^p(\Omega)$.
 Then,
  \begin{enumerate}
   \item[(a)] If $1<p<\frac{n}{2s}$, 
   \[\|u\|_{L^q(\Omega)} \leq C\|g\|_{L^p(\Omega)},~~~q= \frac{np}{n-2ps},\]
   where $C$ is a constant depending only on $n$, $s$, $p$ and $\Lambda_2$.
   \item[(b)] If $p = \frac{n}{2s}$,
   \[\|u\|_{L^q(\Omega)} \leq C\|g\|_{L^p(\Omega)},~~~\forall q <\infty,\]
   where $C$ is a constant depending only on  $n$, $s$, $q$, $\Omega$ and $\Lambda_2$.
   \item[(c)] If $\frac{n}{2s} < p < \infty$,
   \[ \|u\|_{L^\infty(\Omega)} \leq C \|g\|_{L^p(\Omega)},\]
   where $C$ is a constant depending only on  $n$, $s$, $p$, $\Omega$ and $\Lambda_2$.
  \end{enumerate}
\end{prop}

To prove this, we use the results of \cite{GH} to compare the fundamental solution associated to the operator $L$ with the one of the fractional Laplacian.

\subsection{Applications of the main result}

For second order equations, important tools when studying linear and semilinear problems in bounded domains are integration-by-parts-type identities.
For the Laplacian operator, one can easily check that any solution $u$ with $u=0$ on $\partial\Omega$ satisfies
\begin{equation}\label{pohozaev-s=1}
\int_\Omega (x\cdot \nabla u)\Delta u\,dx+\frac{n-2}{2}\int_\Omega u\Delta u\,dx=\frac{1}{2}\int_{\partial\Omega}\left(\frac{\partial u}{\partial\nu}\right)^2(x\cdot\nu)d\sigma.
\end{equation}
This identity was used originally by Pohozaev to prove nonexistence results for semilinear equations $-\Delta u=\lambda f(u)$, and identities of this type are commonly used in the analysis of PDEs.

For nonlocal operators, the second author and Serra found and established the fractional analogue this identity in \cite{RS-P}: if $u$ is any bounded solution to \eqref{eq.ell}, then $u/d^s$ is H\"older continuous in $\overline\Omega$, and
\begin{equation}\label{eq.poh1}
-\int_\Omega (x\cdot \nabla u)(-\Delta)^su\,dx-\frac{n-2s}{2}\int_\Omega u(-\Delta)^su\,dx=\frac{\Gamma(1+s)^2}{2}\int_{\partial\Omega}\left(\frac{u}{d^s}\right)^2(x\cdot\nu)d\sigma.
\end{equation}
Note that the boundary term $u/d^s|_{\partial\Omega}$ has to be understood in the limit sense ---recall that we showed that $u/d^s$ is continuous up to the boundary.

Here, thanks to Theorem \ref{thm.main} and the results in \cite{RS-P}, we are able to prove the Pohozaev identity for solutions to the fractional heat equation \eqref{eq.par}.
The result reads as follows.

\begin{cor}
\label{cor.pohozaev}
 Let $\Omega\subset\R^n$ be any bounded $C^{1,1}$ domain, and $s\in(0,1)$.
 Let $u_0\in L^2(\Omega)$, and let $u(x,t)$ be the solution to the fractional heat equation \eqref{eq.par}.
 Then, for any fixed $t > 0$, the identity \eqref{eq.poh1} holds, with $u = u(\cdot, t)$.
\end{cor}

In \cite{RS-P} the authors prove that the general Pohozaev identity holds whenever a function fulfils certain regularity hypotheses. In particular, they check that these hypotheses are fulfilled by the solutions of the elliptic problem, and so it follows their result. In the present paper we will check that the same hypotheses are also fulfilled by the solutions to the fractional heat equation \eqref{eq.par}, using the results by the second author and Serra, and therefore these functions also fulfil the general Pohozaev identity.
To do this, an essential ingredient is Theorem \ref{thm.main}, since for the right-hand side of the identity to be well-defined we need that $u/\delta^s$ can be extended continuously up to the boundary.

The paper is organized as follows.
In Section \ref{sec3} we prove previous results regarding the existence of eigenfunctions for the Dirichlet elliptic problem and the asymptotic behavior of their eigenvalues.
In Section \ref{sec4} we prove the main result that will be used regarding the regularity of the eigenfunctions.
In Section \ref{sec5} we prove Theorem \ref{thm.main}.
In Section \ref{sec6} we prove Corollary \ref{cor.pohozaev}.

\section{Existence of eigenfunctions and asymptotic behavior of eigenvalues}

\label{sec3}

The results proved in this section hold for general stable operators, with any spectral measure $a \in L^1(S^{n-1})$. 

The aim of this section is to prove the following, referring to the elliptic problem
\begin{equation}
  \label{eq.ell-L}
  \left\{ \begin{array}{rcll}
  L \phi&=&\lambda \phi& \textrm{in }\Omega \\
  \phi&=&0& \textrm{in } \R^n\setminus \Omega.
  \end{array}\right.
\end{equation}

\begin{prop}
\label{prop.asymeiggen}
  Let $\Omega\subset\R^n$ be any bounded domain, and $L$ an operator of the form \eqref{eq.gen1}, with $a\in L^1(S^{n-1})$. 
  Then,
  \begin{enumerate}
   \item[(a)]{ There exist a sequence of eigenfunctions forming a Hilbert basis of $L^2$.
   }
   \item[(b)]{ If $\{\lambda_k\}_{k\in\N}$ is the sequence of eigenvalues associated to the eigenfunctions of $L$ in increasing order, then
    \[
      \lim_{k\to \infty} \lambda_k k^{-\frac{2s}{n}} = C_0,
    \]
    for some constant $C_0$ depending only on $n$, $s$, $\Omega$ and $L$. Moreover,
    \[
     C(\mu_2) \leq C_0 \leq C(\mu_1),
    \]
    for $C(\mu_1)$ and $C(\mu_2)$ positive constants depending only on $n$, $s$, $\Omega$ and $\mu_1$ and $\mu_2$ respectively; where $\mu_1>0$ and $\mu_2>0$ are given by the expressions
    \begin{equation}\label{eq.defmu}
      \mu_2 = \int_{S^{n-1}} a(\theta) d\theta,~~~~~~\mu_1 = \inf_{\nu\in S^{n-1}}\int_{S^{n-1}}|\nu\cdot\theta|^{2s}a(\theta) d\theta.
    \end{equation}

   }
  \end{enumerate}
\end{prop}

For the fractional Laplacian, the result in Proposition \ref{prop.asymeiggen} is well known, and was already proved in the sixties; see \cite{BG} and also \cite{FG}.
 
For general stable operators $L$, the asymptotic behavior of its eigenvalues follows from the following result by Geisinger \cite{G}.

Given an operator $L$, we denote by $A(\xi)$ its Fourier symbol (i.e., $\widehat{L u}(\xi) = A(\xi)\hat u(\xi)$). 
For the fractional Laplacian, we recall that $A(\xi) = |\xi|^{2s}$. 

Let us state the result by Geisinger, first introducing two hypotheses on $A$.

\begin{enumerate}
 \item {There is a function $A_0:\R^n\to \R$ with the following three properties. 
 $A_0$ is homogeneous of degree $\alpha>0$: $A_0(\nu\xi) = \nu^\alpha A_0(\xi)$ for $\xi\in \R^n$ and $\nu > 0$. 
 The set of $\xi\in\R^n$ with $A_0(\xi) < 1$ has finite Lebesgue-measure, and the function $A_0$ fulfils
 \[
  \lim_{\nu\to\infty}\nu^{-\alpha}A(\nu\xi) = A_0(\xi),
 \]
with uniform convergence in $\xi$.
 }
 \item{
 There are constants $C_*>0$ and $M\in\N$ such that for all $\eta\in\R^n$,
 \[
  \sup_{\xi\in\R^n}\left(\frac{1}{2}(A(\xi+\eta)+A(\xi-\eta))-A(\xi)\right)\leq C_*(1+|\eta|)^M.
 \]
 }
\end{enumerate}

Under these two conditions, we have the following.

\begin{thm}[\cite{G}]
\label{teo.geisinger}
 Let $\Omega\subset\R^n$ be an open set of finite volume and assume that $A$ is the symbol of a differential operator $L$ that satisfies the previous two conditions. 
 Then,
 \[
  \lim_{k\to \infty} \lambda_k k^{-\frac{2s}{n}} = C_{L,\Omega},
 \]
 where
 \[
  C_{L,\Omega} = (2\pi)^{2s}|\Omega|^{-2s/n}V_L^{-2s/n},
 \]
 and
 \[
  V_L = \left| \left\{ \xi \in \R^n: A_0(\xi) < 1 \right\} \right|.
 \]
\end{thm}

To prove Proposition \ref{prop.asymeiggen}, we only have to check that any stable operator \eqref{eq.gen1} fulfils the previous conditions.
To do so, we use that the Fourier symbol $A(\xi)$ of $L$ can be explicitly written in terms of $s$ and the spectral measure $a$, as
\[
 A(\xi) = \int_{S^{n-1}}|\xi\cdot \theta|^{2s}a(\theta)d \theta;
\]
see for example \cite{S}. 

Using this expression, it is clear that
\[
 0 < \mu_1|\xi|^{2s} \leq A(\xi) \leq \mu_2|\xi|^{2s},
\]
where $\mu_1$ and $\mu_2$ are given by \eqref{eq.defmu}.
Notice that these constants are strictly positive for any stable operator with $a\in L^1(S^{n-1})$.

We can now proceed to prove Proposition \ref{prop.asymeiggen}.

\begin{proof}[Proof of Proposition \ref{prop.asymeiggen}]
The first statement is well known, and follows, for example, seeing that the norms $H^s$ and $H^s_L$ are equivalent in $\R^n$, which can be done easily in the Fourier side. Here, $H^s_L$ stands for the $H^s$ norm associated to the operator $L$, given by:
 \[
  \|u\|_{H^s_L} := \|u\|_{L^2} + [u]_{H^s_L},~~~~[u]_{H^s_L}^2 := \frac{1}{2}\int_{\R^n}\int_{\R^n}\frac{(u(x)-u(x+y))^2}{|y|^{n+2s}} a\left(\frac{y}{|y|}\right) dxdy,
 \]
 where $[u]_{H^s_L}$ is the Gagliardo seminorm associated to the operator $L$. In the Fourier side, it can we written as
 \[
  [u]_{H^s_L}^2 = \int_{\R^n}A(\xi){\hat u}^2(\xi)d\xi.
 \]

For the second statement we use Theorem \ref{teo.geisinger}.

Indeed, let $A$ be the Fourier symbol of the stable operator $L$. 
It is enough to check the two conditions of the theorem by Geisinger. 
The first condition trivially holds taking $A_0 = A$, since it is homogeneous, and we have the previous bounds.

We need to check the second condition. 
To do so we claim that, for any $a \geq b\geq 0$ and $s\in(0,1)$, we have
 \begin{equation}\label{claim}
  2a^{2s}+2b^{2s} \geq (a+b)^{2s}+(a-b)^{2s}.
 \end{equation}
Indeed, since $s\in(0,1)$, by concavity we have
 \begin{align*}
  (a+b)^{2s}+(a-b)^{2s} = & (a^2+b^2+2ab)^s+(a^2+b^2-2ab)^s\\
  \leq & 2(a^2+b^2)^s\\
  \leq & 2(a^{2s}+b^{2s}),
 \end{align*}
 as claimed.

Thus, using \eqref{claim}, we find
\[ |\xi\cdot\theta + \eta\cdot \theta|^{2s} + |\xi\cdot\theta - \eta\cdot \theta|^{2s} \leq 2 |\xi\cdot\theta|^{2s}+2|\eta\cdot\theta|^{2s},\]
and therefore,
\[\begin{split}
  A(\xi+\eta)+&A(\xi-\eta)-2A(\xi) = \\
    &= \int_{S^{n-1}} \left\{  |\xi\cdot\theta + \eta\cdot \theta|^{2s} + |\xi\cdot\theta - \eta\cdot \theta|^{2s}-2|\xi\cdot\theta|^{2s}  \right\} a(\theta)d\theta\\
    &\leq  2 \int_{S^{n-1}} |\eta\cdot\theta|^{2s} a(\theta)d\theta \leq  2 |\eta|^{2s}\mu_2,
  \end{split}\]
as desired. 

Finally, we know that
  \[
    C_0 = (2\pi)^{2s}|\Omega|^{-\frac{2s}{n}}\left| \left\{ \xi \in \R^n: A(\xi) < 1 \right\} \right|^{-\frac{2s}{n}},
  \]
and hence, using the bounds on $A$, we obtain
  \[
    (2\pi)^{2s}|\Omega|^{-\frac{2s}{n}} V_n(\mu_2^{-\frac{1}{2s}}) \leq C_0 \leq (2\pi)^{2s}|\Omega|^{-\frac{2s}{n}} V_n(\mu_1^{-\frac{1}{2s}}),
  \]
where $V_n(R)$ denotes the volume of an $n$-ball with radius $R$.
\end{proof}

\section{Regularity of eigenfunctions}
\label{sec4}

The aim of this section is to prove the following.

\begin{prop} \label{prop.eiglinf}
Let $\Omega\subset \R^n$ be any bounded domain, $s\in(0,1)$, and $L$ an operator of the form \eqref{eq.gen1}-\eqref{eq.gen2}.
 
Let $\phi$ be any solution to
\begin{equation} \label{eq.rs1.gen}
    \left\{ \begin{array}{rcll}
    L \phi&=&\lambda \phi& \textrm{in }\Omega\\
    \phi&=&0& \textrm{in } \R^n\setminus \Omega,
\end{array}\right.
\end{equation}
Then, $\phi\in L^\infty(\Omega)$, and
  \[
  \|\phi\|_{L^\infty(\Omega)}\leq  \ C\lambda^{w-1}\|\phi\|_{L^2(\Omega)},
  \]
for some constant $C$ depending only on $n$, $s$, $\Omega$ and the ellipticity constant $\Lambda_2$, and some $w\in\N$ depending only on $n$ and $s$.
\end{prop}

\begin{rem}
When $\Omega$ and the spectral measure $a$ are $C^\infty$, a very similar result has been found recently in \cite[Theorem 2.3]{Gr}, which follows using the results of \cite{Gr2}.
\end{rem}

Combining Proposition \ref{prop.eiglinf} with the boundary regularity results in \cite{RS-F}, we will obtain the following.

\begin{cor} \label{cor.reg-eig}
Let $\Omega\subset \R^n$ be any bounded $C^{1,1}$ domain, $s\in(0,1)$, and $L$ an operator of the form \eqref{eq.gen1}-\eqref{eq.gen2}.

Let $\phi$ be any solution to \eqref{eq.rs1.gen}. 
Then, $\phi\in C^s(\R^n)$ and $\phi/\delta^s \in C^{s-\epsilon}(\overline{\Omega})$ for any $\epsilon > 0$, with the estimates
  \begin{align*}
  \|\phi\|_{C^s(\R^n)}\leq & \ C\lambda^w\|\phi\|_{L^2(\Omega)} ,\\
  \|\phi/\delta^s\|_{C^{s-\epsilon}(\overline{\Omega})}\leq & \ C\lambda^w\|\phi\|_{L^2(\Omega)}.
  \end{align*}
The constant $C$ depends only on $n$, $s$, $\Omega$, $\epsilon$ and the ellipticity constants $\Lambda_1$ and $\Lambda_2$, and $w\in\N$, depends only on $n$ and $s$.
\end{cor}

Proposition \ref{prop.eiglinf} will follow from Proposition \ref{prop.maingen} and a bootstrap argument. 
We next prove Proposition \ref{prop.maingen}.

The following lemma will be essential.
It is an immediate consequence of the results of \cite{GH}.

\begin{lem}\label{lemma.maingen}
 Let $s \in (0,1)$ and $n > 2s$. 
 Let $L$ be any operator of the form \eqref{eq.gen1}-\eqref{eq.gen2}. 
 
 Then, there exists a function $V$ which is fundamental solution of $L$ in $\R^n$.
 Namely, for all $g\in L^1(\R^n)$ with appropriate decay at infinity, the function
 \[u(x) = \int_{\R^n} g(y)V(x-y) dy\]
 satisfies
 \[Lu = g\textrm{ in } \R^n.\]

In addition, the function $V$ satisfies the bound
\[
  V(x) \leq \frac{c_2}{|x|^{n-2s}}
\]
for some $c_2$ positive constant depending only on $n$, $s$ and $\Lambda_2$.
\end{lem}

\begin{proof}
It was proved in \cite{GH} that the heat kernel of the operator $L$ satisfies
\[p(x,1)\leq \frac{C}{1+|x|^{n+2s}}\]
for all $x\in \R^n$ and $t>0$.
Moreover, we also have the scaling property
\[p(x,t)=t^{-\frac{n}{2s}}\,p(t^{-\frac{1}{2s}}x,\,1).\]

Then, since $n>2s$, the function
\[V(x)=\int_0^{+\infty}p(x,t)dt\]
is well defined, and it satisfies 
\[V(x)=|x|^{2s-n}V\left(\frac{x}{|x|}\right).\]
Finally, it also follows from the upper bound on $p(x,t)$ that for $|x|=1$ we have $V(x)\leq C$, and thus the lemma follows.
\end{proof}

By the previous lemma, we can define $L^{-1}$ as the operator
\[L^{-1}g(x) = \int_{\R^n} g(y)V(x-y) dy.\]
For this operator, we have the following.

\begin{lem}\label{lemma.gen0}
 Let $s \in (0,1)$ and $n > 2s$.
 Let $L$ be any operator of the form \eqref{eq.gen1}-\eqref{eq.gen2}.
 
 Let $g$ and $u$ be such that $u = L^{-1}g$ in $\R^n$.
 Suppose $g\in L^p(\R^n)$, with $1\leq p < \frac{n}{2s}$. 
 Then, 
  \[
   \|u\|_{L^q(\R^n)} \leq C \|g\|_{L^p(\R^n)},~~~~q = \frac{np}{n-2ps},
  \]
  for some constant $C$ depending only on $n$, $s$, $p$ and $\Lambda_2$.
\end{lem}

\begin{proof}
Let us first recall the well known Hardy-Littlewood-Sobolev inequality.
It states that, if $f\in L^p(\R^n)$, then
\[ \|I_{2s}f\|_{L^q(\R^n)} \leq C\|f\|_{L^p(\R^n)},~q = \frac{np}{n-2ps},\]
for some constant $C$ depending only on $p$, $n$ and $s$. 
Here, $I_{2s}f$ stands for the Riesz potential, defined by
\[ (I_{2s}f)(x) = C_{n,s}\int_{\R^n}\frac{f(y)}{|x-y|^{n-2s}}dy.\]

We now bound $|u|$ by some expression equivalent to a Riesz potential of $g$, in order to use the Hardy-Littlewood-Sobolev inequality. 
Since
\[ u(x) = L^{-1}g(x) = \int_{\R^n}g(y) V(x-y) dy\]
and $0\leq V(x-y) \leq  \frac{c_2}{|x-y|^{n-2s}}$, we have
\[
 |u(x)| \leq \int_{\R^n}|g(y)| V(x-y) dy \leq \int_{\R^n}|g(y)|\frac{c_2}{|x-y|^{n-2s}}dy = C (I_{2s} |g|)(x).
\]
Thus, we find
\[
 \|u\|_{L^q(\R^n)}\leq C' \|I_{2s} |g|\|_{L^q(\R^n)} \leq C\|g\|_{L^p(\R^n)},~q = \frac{np}{n-2ps},
\]
as we wanted to see.
\end{proof}

We can now prove Proposition \ref{prop.maingen} for a general operator $L$ of the form \eqref{eq.gen1} and \eqref{eq.gen2}.

\begin{proof}[Proof of Proposition \ref{prop.maingen}]
First, notice that if $n \leq 2s$, then $n = 1$.
In this case any stable operator \eqref{eq.gen1} is the fractional Laplacian, for which Proposition \ref{prop.maingen} is known (see for example \cite{RS-ext}).
Thus, from now on we assume that $n>2s$.

(a) Consider the problem
 \[
  Lv = |g|\textrm{ in }\R^n,
 \]
 where $g$ has been extended from $\Omega$ to $\R^n$ by zero. 
 We know that there is a $v \geq 0$ solving the problem, since we can define
 \[
  v(x) = \int_\Omega |g(y)|V(x-y) dy \geq 0.
 \]

Now, using the comparison principle we have $-v \leq u \leq v$, since $|g| \geq g \geq -|g|$. This means $\|v\|_{L^q(\R^n)} \geq \|u\|_{L^q(\Omega)}$, and by Lemma \ref{lemma.gen0}  we have that
 \[
   \|v\|_{L^q(\R^n)} \leq C \|g\|_{L^p(\R^n)},~~~~q = \frac{np}{n-2ps},
 \]
 just like we wanted to see.

 (b) It follows using part (a), and making $p\uparrow \frac{n}{2s}$ using that
 \[
  \|g\|_{L^p(\Omega)} \leq C\|g\|_{L^r(\Omega)},~~r = \frac{n}{2s},
 \]
 by H\"older's inequality on bounded domains.
 
(c) Define $v$ as before, so that using H\"older's inequality
 \[
  0 \leq |u(x)| \leq v(x) = \int_\Omega |g(y)|V(x-y) dy \leq \|g\|_{L^p(\Omega)}\left(\int_\Omega V(x-y)^{p'} dy\right)^{1/p'},
 \]
 where $\frac{1}{p}+\frac{1}{p'} = 1$. We now want to bound $\int_\Omega V(x-y)^{p'} dy$. Consider $B_R(0)$ a ball centred at the origin such that $\Omega \subset B_R$ (exists because $\Omega$ is bounded). Then we want to see whether the following integral converges,
 \[
  \int_{B_R(0)} V(y)^{p'} dy,
 \]
 which will be enough by seeing that
  \[
  \int_{B_R(0)} \frac{1}{|y|^{(n-2s)p'}} dy < \infty,
 \]
 and this last integral converges because $p > \frac{n}{2s}$.

 Therefore, we reach
 \[
    \|u\|_{L^\infty(\Omega)} \leq C \|g\|_{L^p(\Omega)},
 \]
 for some constant $C$ depending only on $n$, $s$, $p$, $\Lambda_2$ and $\Omega$.
\end{proof}

Finally, we prove next Proposition \ref{prop.eiglinf} and Corollary \ref{cor.reg-eig}.

\begin{proof}[Proof of Proposition \ref{prop.eiglinf}]

 In order to prove this result, consider the problem \eqref{eq.propmaingen}, with $g = \lambda\phi$, being $\phi$ an eigenfunction and $\lambda$ its eigenvalue. We already know that $\phi\in L^2(\Omega)$. If $\frac{n}{2s} < 2$, apply Proposition \ref{prop.maingen} third result, (c), with $p = 2$ to get
\[
\|\phi\|_{L^\infty(\Omega)}\leq C\lambda \|\phi\|_{L^2(\Omega)}.
\]

 If $\frac{n}{2s} = 2$, first use the second result, (b), from Proposition \ref{prop.maingen} and consider $q > \frac{n}{2s}$, which reduces to the previous case. Therefore, in this situation $w = 3$.

 Suppose $\frac{n}{2s} > 2$. Apply Proposition \ref{prop.maingen}, first result, (a).
 \[
  \|\phi\|_{L^q(\Omega)} \leq C \lambda \|\phi\|_{L^p(\Omega)},~~~q= \frac{np}{n-2ps}.
 \]

 Now the constant $C$ does not depend on $\lambda$. If $\phi\in L^p(\Omega)$, then $\phi\in L^q(\Omega)$ (for $2\leq p  <\frac{n}{2s}$). Take an initial $p_0 = 2$.

 Define $p_{k+1} = \frac{np_k}{n-2p_k s}$. This sequence is obviously increasing, has no fixed points while $n > 2p_ks$ and implies the following chain of inequalities,
 \[
 C^{(0)}\|\phi\|_{L^{p_0}(\Omega)} \geq C^{(1)}\lambda^{-1}\|\phi\|_{L^{p_1}(\Omega)} \geq \dots \geq C^{(k+1)}\lambda^{-k-1}\|\phi\|_{L^{p_{k+1}}(\Omega)}.
 \]

As long as $n > 2p_ks$. Define $N$ as the index of the first time $n \leq 2p_N s$. We know that $\phi \in L^{p_N}(\Omega)$, with $p_N \geq \frac{n}{2s}$. If $p_N = \frac{n}{2s}$, consider $p_N' = p_N-\epsilon$, then $q = \frac{np_N'}{n-2p_N's}$ is, for some $\epsilon > 0$, larger than $\frac{n}{2s}$. It is possible to conclude that $\phi \in L^Q(\Omega)$, for some $Q > \frac{n}{2s}$.

 Now use the third result from Proposition \ref{prop.maingen}, to see $\phi \in L^\infty(\Omega)$.
 \[
 \|\phi\|_{L^\infty(\Omega)}\leq C\lambda \|\phi\|_{L^{p_N}(\Omega)},
 \]
 which by the previous chain of inequalities implies
 \[
  \|\phi\|_{L^\infty(\Omega)} \leq C \lambda^{w-1} \|\phi\|_{L^2(\Omega)},
 \]
 where $w = N+2$ if $p_N > \frac{n}{2s}$ and $w = N+3$ if $p_N = \frac{n}{2s}$; as we wanted to see.
\end{proof}

%
%
%
%

\begin{proof}[Proof of Corollary \ref{cor.reg-eig}]
It follows using Proposition \ref{prop.eiglinf} and \cite{RS-F}.
\end{proof}

\section{Proof of Theorem \ref{thm.main}}
\label{sec5}

The aim of this section is to prove the main result of the present paper, Theorem \ref{thm.main}.
After that, we will also give a Corollary concerning the regularity up to the boundary of the solution $u(x,t)$ in the $t$-variable; see Corollary \ref{cor.main}.

\begin{proof}[Proof of Theorem \ref{thm.main}]
Let us prove (b). The result corresponding to (a) will follow proceeding exactly the same way.

First, by Proposition \ref{prop.asymeiggen} the solution $u$ can be expressed as
 \[
  u(x, t) = \sum_{k>0} u_k\phi_k e^{-\lambda_k t},
 \]
 which follows immediately by separation of variables and imposing the initial condition $u_0 = \sum_{k>0} u_k\phi_k$. Here, $\phi_k$ denote the unitary eigenfunctions of the Dirichlet elliptic problem, and $\lambda_k$ the corresponding eigenvalues, in increasing order.

 Notice that we already know that the solution will always be in $L^2(\Omega)$, since the $L^2$ norm is decreasing with time.

 We then try to bound $\|u(\cdot, t)/\delta^s\|_{C^{s-\epsilon}(\overline\Omega)}$ for any fixed $\epsilon > 0$, through the expression found in Corollary \ref{cor.reg-eig}, and noticing that the sequence $|u_k|$ has a maximum (since it converges to 0) and it satisfies $\max_{k>0} |u_k| \leq \|u_0\|_{L^2(\Omega)}$,
 \begin{align*}
  \left\|\frac{u(\cdot, t)}{\delta^s}\right\|_{C^{s-\epsilon}(\overline\Omega)} = & \ \left\| \sum_{k>0} u_k\frac{\phi_k}{\delta^s} e^{-\lambda_k t}\right\|_{C^{s-\epsilon}(\overline\Omega)}\\
  \leq & \ \sum_{k>0} |u_k|\left\|\frac{\phi_k}{\delta^s}\right\|_{C^{s-\epsilon}(\overline\Omega)} e^{-\lambda_k t} \\
  \leq & \ \|u_0\|_{L^2(\Omega)} \sum_{k>0} C_{\lambda_k} \left\|\phi_k\right\|_{L^2(\Omega)} e^{-\lambda_k t} \\
  = & \ \|u_0\|_{L^2(\Omega)} C \sum_{k>0} \lambda_k^w e^{-\lambda_k t}.
 \end{align*}

 We have bounded $\|u(\cdot, t)/\delta^s\|_{C^{s-\epsilon}(\overline\Omega)}$ by an expression decreasing with time (since $\lambda_k > 0$ always), and where $C$ depends only on $n$, $s$, $\Omega$, $\Lambda_1$, $\Lambda_2$ and $\epsilon$. Therefore, it is only needed to consider $\|u(\cdot, t_0)/\delta^s\|_{C^{s-\epsilon}(\overline\Omega)}$ and bound it using the previous expression.

 So it is enough to prove the convergence of the series $\sum_{k>0} \lambda_k^w e^{-\lambda_k t_0}$ considering the asymptotics previously introduced. Let us study the convergence of the tail.

 There exists $k_0$ such that
 \begin{align*}
 \sum_{k\geq k_0} \lambda^w_k e^{-\lambda_k t_0}  < & \sum_{k\geq k_0} \left(\frac{3}{2}C_0 k^{\frac{2s}{n}}\right)^w e^{-\frac{1}{2}C_0k^{\frac{2s}{n}} t_0} \\
 = & \left(\frac{3}{2}\right)^w C_0^w \sum_{k\geq k_0} k^{\gamma w} e^{-\frac{1}{2}C_0 k^{\gamma} t_0},
 \end{align*}
 where it has been used the notation and the results from Proposition \ref{prop.asymeiggen}. The notation has been simplified, $\gamma = \frac{2s}{n}$.

 As well as that, the qualitative convergence with respect to $t_0$, as $t_0$ goes to 0, is the same in the last and in the first expression. In order to check the convergence  we use the integral criterion, changing variables $y = x^\gamma$, and $z = \frac{1}{2}C_0 yt_0$; and defining $\beta := w + \frac{1-\gamma}{\gamma} = w + \frac{n}{2s} - 1$,
  \begin{align*}
 \int_{ k_0}^\infty x^{\gamma w} e^{-\frac{1}{2}C_0 x^{\gamma} t_0} dx = & \frac{1}{\gamma}\int_{ k_0^\gamma}^\infty y^{w+\frac{1-\gamma}{\gamma}} e^{-\frac{1}{2}C_0 yt_0}dy \\
 = & \frac{2^{\beta+1}}{\gamma C_0^{\beta + 1} t_0^{\beta+1}}\int_{\frac{1}{2}C_0 t_0 k_0^\gamma}^\infty z^\beta e^{-z} dz < C'(\mu_2) < +\infty,
  \end{align*}
for some constant $C'(\mu_2)$ depending on $n$, $s$, $\Omega$ and $\mu_2$ (which can be bound by an expression depending on $\Lambda_2$). Notice that it has been possible to bound the expression since $t_0 > 0$. For $t_0= 0$ the previous procedure is not appropriate, which is consistent with the fact that the initial condition does not fulfil that $u_0/\delta^s$ is in $C^{s-\epsilon}(\overline\Omega)$. That is, fixed $t_0 > 0$,
 \[
  \|u(\cdot, t_0)/\delta^s\|_{C^{s-\epsilon}(\overline\Omega)} \leq C(t_0) \|u_0\|_{L^2(\Omega)},
 \]
 where $C(t_0)$ depends only on $n$, $s$, $\Omega$, $\epsilon$, the ellipticity constants $\Lambda_1$ and $\Lambda_2$ and $t_0$. The dependence on $t_0$ has been found before, as $t_0$ approaches 0: $C(t_0) = O(t_0^{-w-\frac{n}{2s}})$, for $t_0 \downarrow 0$.

 Finally, to prove (a), we use the same argument as used in (b), but now consider
 \[
  \left\|u(\cdot, t)\right\|_{C^{s}(\R^n)} =  \left\| \sum_{k>0} u_k\phi_k e^{-\lambda_k t}\right\|_{C^{s}(\R^n)},
 \]
 and follow the same way, to see that $\left\|u(\cdot, t)\right\|_{C^{s}(\R^n)}$ is bounded by an expression equivalent to the one found for (b).
\end{proof}

At this point we can give the following result.

\begin{cor} \label{cor.main}
 The solution to the fractional heat equation for general nonlocal operators \eqref{eq.par-L}, $u$, where $L$ is of the form \eqref{eq.gen1} and \eqref{eq.gen2} for an initial condition at $t = 0$, $u(x, 0) = u_0(x)\in L^2(\Omega)$, and for a $C^{1,1}$ bounded domain $\Omega$, satisfies that
 \begin{enumerate}
  \item For each $t_0 > 0$,
    \[
    \sup_{t \geq t_0} \left\| \frac{\de^j u}{\de t^j} (\cdot, t) \right\|_{C^s(\R^n)} \leq C_1^{(j)}\|u_0\|_{L^2(\Omega)},~\forall j\in \N
  \]
  \item For each $t_0 > 0$,
    \[
    \sup_{t \geq t_0} \left\| \frac{1}{\delta^s}\frac{\de^j u}{\de t^j} (\cdot, t) \right\|_{C^{s-\epsilon}(\overline{\Omega})} \leq C_2^{(j)}\|u_0\|_{L^2(\Omega)},~\forall j\in \N
  \]
   where $\delta(x) = \textrm{dist}(x, \de \Omega)$, and for any $\epsilon > 0$.
 \end{enumerate}
The constants $C_1^{(j)}$ and $C_2^{(j)}$ depend only on $t_0$, $n$, $s$, $j$, the ellipticity constants $\Lambda_1$ and $\Lambda_2$ and $\Omega$ (and $C_2^{(j)}$ also on $\epsilon$), and blow up as $t_0 \downarrow 0$.
\end{cor}

\begin{proof}
 To prove this corollary we can simply use the same argument used in the proof of Theorem \ref{thm.main}, but now considering the expression of $\frac{\de^j u}{\de t^j} (\cdot, t)$ obtained deriving term by term,
  \[
     \frac{\de^j u}{\de t^j} (\cdot, t) = (-1)^j\sum_{k> 0}\lambda_k^j u_k\phi_k(x)e^{-\lambda_k t}.
  \]

  The temporal derivation term by term can be done because the series expansion is uniformly convergent with respect to time over compact subsets in $\R^+$.

  Now proceed with the proof of Theorem \ref{thm.main} using $w+j$ instead of $w$.
\end{proof}

\section{The Pohozaev identity for solutions to the fractional heat equation}
\label{sec6}

In this section we prove Corollary \ref{cor.pohozaev}.

The main result we will need, corresponding to \cite[Proposition 1.6]{RS-P}, is the following. 
We define $\Omega_\rho = \{x\in \Omega : \delta(x) \geq \rho\}$.

\begin{thm}[\cite{RS-P}]
\label{thm.pohozaevRS}
 Let $\Omega$ be a bounded and $C^{1,1}$ domain. Assume that $u$ is a $H^s(\R^n)$ function which vanishes in $\R^n\setminus\Omega$, and satisfies
 \begin{enumerate}
  \item[(a)]{
  $u \in C^s(\R^n)$ and, for every $\beta \in [s, 1+2s)$, $u$ is of class $C^\beta(\Omega)$ and
  \[
   [u]_{C^\beta(\Omega_\rho)} \leq C\rho^{s-\beta},~\forall \rho \in (0,1).
  \]
  }
  \item[(b)]{
  The function $u/\delta^s|_\Omega$ can be continuously extended to $\overline{\Omega}$. Moreover, there exists $\alpha \in (0,1)$ such that $u/\delta^s\in C^\alpha(\overline{\Omega})$. In addition, for all $\beta\in [\alpha, s+\alpha]$, it holds the estimate
  \[
   [u/\delta^s]_{C^\beta(\Omega_\rho)} \leq C\rho^{\alpha -\beta},~\forall \rho\in(0,1).
  \]
  }
  \item[(c)] $\fls u$ is bounded in $\Omega$.
 \end{enumerate}

 Then, the following identity holds
 \[
  \int_\Omega (x\cdot \nabla u) \fls u dx = \frac{2s-n}{2}\int_\Omega u \fls u dx - \frac{\Gamma(1+s)^2}{2} \int_{ \de \Omega}\left(\frac{u}{\delta^s}\right)^2 (x\cdot \nu) d\sigma,
 \]
  where $\nu$ is the unit outward normal to $\de \Omega$ at $x$, and $\Gamma$ is the Gamma function.
\end{thm}

By the results in \cite{RS}, any bounded solution to the elliptic problem $(-\Delta)^s u=f(x,u)$ in $\Omega$ satisfies the hypotheses of the previous result. 
Using the results in \cite{RS}, we show here that the same happens for solutions to the fractional heat equation $\partial_t u+(-\Delta)^su=0$ in $\Omega$.

\begin{proof}[Proof of Corollary \ref{cor.pohozaev}]
 In Theorem \ref{thm.main} we have checked that indeed $u(\cdot, t)/\delta^s$ can be extended continuously up to the boundary for positive times, so that the boundary integral makes sense.

 For the fixed $t > 0$, we will use Theorem \ref{thm.pohozaevRS} to prove the identity. It is enough to check that the hypothesis are fulfilled. Firstly, $u$ for positive times is always in $H^s(\R^n)$, since it is a solution of the fractional heat equation. In addition, to check the third hypothesis we could see that $\de_t u = -\fls u$ and in Corollary \ref{cor.main} we had shown that all temporal derivatives are $C^s(\R^n)$ for positive times, so in particular, are uniformly bounded in $\Omega$.

 In order to check hypothesis (a) and (b), we will use the result \cite[Corollary 1.6]{RS}, where there is stated that any solution of the elliptic problem $(-\Delta)^s u=f(x,u)$ in $\Omega$, for $f \in C^{0,1}_{loc}(\overline\Omega\times \R)$, satisfies the hypotheses, for $\alpha$ and $C$ depending only on $s$, $\Omega$, $f$, $\|u\|_{L^\infty(\R^n)}$ and $\beta$. Moreover, from the proof of this result one can see that if $f(x, u) = \lambda u$, then $C = \bar C\lambda \|u\|_{L^\infty(\Omega)}$, where $\bar C$ depends only on $\Omega$, $s$, $n$ and $\beta$. Therefore, considering the eigenvalue-eigenfunction problem we have,
 \[
  [\phi_k]_{C^\beta(\Omega_\rho)} \leq C\lambda_k\|\phi_k\|_{L^\infty(\Omega)}\rho^{s-\beta},~\forall \rho \in (0,1),~\beta \in [s, 1+2s),
 \]
 and
 \[
  [\phi_k/\delta^s]_{C^\beta(\Omega_\rho)}\leq C\lambda_k\|\phi_k\|_{L^\infty(\Omega)}\rho^{\alpha - \beta},~\forall \rho \in (0,1),~\beta \in [\alpha,s+\alpha],
 \]
where $\phi_k$ is the $k$-th eigenfunction of the fractional Laplacian problem in $\Omega$, and $C$ depends only on $n$, $s$, $\Omega$ and $\beta$. Proceed like in the proof of Theorem \ref{thm.main}, expressing $u(x, t) = \sum_{k > 0} u_k \phi_k e^{-\lambda_k t}$, being $u_k$ the coefficients of $u_0$ in the basis $\{\phi_k\}_{k > 0}$,
\[
 [u(\cdot, t)]_{C^\beta(\Omega_\rho)} \leq \sum_{k > 0} u_k  [\phi_k]_{C^\beta(\Omega_\rho)} e^{-\lambda_k t} \leq C\rho^{s-\beta}\sum_{k > 0} u_k \lambda_k\|\phi_k\|_{L^\infty(\Omega)} e^{-\lambda_k t},
\]
similarly,
\[
 [u(\cdot, t)/\delta^s]_{C^\beta(\Omega_\rho)} \leq C\rho^{\alpha-\beta}\sum_{k > 0} u_k \lambda_k\|\phi_k\|_{L^\infty(\Omega)}e^{-\lambda_k t}.
\]

Now, we recall the result from Proposition \ref{prop.eiglinf}, which stated
\[
 \|\phi_k\|_{L^\infty(\Omega)} \leq C\lambda_k^{w-1}\|\phi_k\|_{L^2(\Omega)}
\]
for some $w\in \N$, and $C$ depending only on $n$, $s$ and $\Omega$. Hence, assuming unitary eigenfunctions,
\[
 [u(\cdot, t)]_{C^\beta(\Omega_\rho)} \leq C\rho^{s-\beta}\sum_{k > 0} u_k \lambda_k^w e^{-\lambda_k t}
\]
and
\[
 [u(\cdot, t)/\delta^s]_{C^\beta(\Omega_\rho)} \leq C\rho^{\alpha-\beta}\sum_{k > 0} u_k \lambda_k^w e^{-\lambda_k t}.
\]

Which are expressions almost identical to the ones found in the proof of Theorem \ref{thm.main}. Proceeding the same way, we reach
\[
 [u(\cdot, t)]_{C^\beta(\Omega_\rho)} \leq C(t)\|u_0\|_{L^2(\Omega)}\rho^{s-\beta},
\]
and
\[
 [u(\cdot, t)/\delta^s]_{C^\beta(\Omega_\rho)} \leq C(t)\|u_0\|_{L^2(\Omega)}\rho^{\alpha-\beta},
\]
for some constant $C(t)$ depending only on $n$, $s$, $\Omega$, $\beta$ and $t$, that blows up when $t \downarrow 0$. So, for any fixed $t > 0$, we have seen that hypothesis of Corollary \ref{cor.pohozaev} are fulfilled, and therefore,
 \[
  \int_\Omega (x\cdot \nabla u) \fls u dx = \frac{2s-n}{2}\int_\Omega u \fls u dx - \frac{\Gamma(1+s)^2}{2} \int_{ \de \Omega}\left(\frac{u}{\delta^s}\right)^2 (x\cdot \nu) d\sigma,
 \]
 where we have now again used that $u = u(\cdot, t)$ for a fixed $t > 0$.
\end{proof}


\begin{thebibliography}{00}

\bibitem{BG} R. M. Blumenthal, R. K. Getoor, \emph{The asymptotic distribution of the eigenvalues for a class of Markov operators}, Pacific J. Math. 9 (1959), 399-408.

\bibitem{BG2} R. M. Blumenthal, R. K. Getoor, D. B. Ray, \emph{On the distribution of first hits for the symmetric stable processes}, Trans. Amer. Math. Soc. 99 (1961), 540-554.

\bibitem{BGR} K. Bogdan, T. Grzywny, M. Ryznar, \emph{Heat kernel estimates for the fractional Laplacian with Dirichlet conditions}, Ann. of Prob. 38 (2010), 1901-1923.

\bibitem{CCV} L. Caffarelli, C. H. Chan, A. Vasseur, \emph{Regularity theory for parabolic nonlinear integral operators}, J. Amer. Math. Soc. 24 (2011), 849-869.

\bibitem{CV} L. Caffarelli, J. L. Vazquez, \emph{Nonlinear porous medium flow with fractional potential pressure}, Arch. Rat. Mech. Anal. 202 (2011), 537-565.

\bibitem{CSV} L. Caffarelli, F. Soria and J. L. Vázquez \emph{Regularity of solutions of the fractional porous medium flow}, to appear in Journal Europ. Math. Society.

\bibitem{CVasseur} L. Caffarelli, A. Vasseur, \emph{Drift diffusion equations with fractional diffusion and the quasi-geostrophic equation}, Ann. of Math. 171 (2010), 1903-1930.

\bibitem{CF} L. Caffarelli, A. Figalli, \emph{Regularity of solutions to the parabolic fractional obstacle problem}, J. Reine Angew. Math., 680 (2013), 191-233.

\bibitem{CD} H. Chang-Lara, G. Davila, \emph{Regularity for solutions of nonlocal parabolic equations}, Calc. Var. Partial Differential Equations 49 (2014), 139-172.

\bibitem{CD2} H. Chang-Lara, G. Davila, \emph{Regularity for solutions of nonlocal parabolic equations II}, J. Differential Equations 256 (2014), 130-156.

\bibitem{CKS} Z. Chen, P. Kim, R. Song, \emph{Heat kernel estimates for the Dirichlet fractional Laplacian}, J. Eur. Math. Soc. 12 (2010), 1307-1329.

\bibitem{dPQR} A. De Pablo, F. Quirós, A. Rodríguez and J. L. Vázquez \emph{A general fractional porous medium equation} Comm. Pure Applied Mathematics, 65 (2012), 1242-1284.

\bibitem{DPV} E. Di Nezza, G. Palatucci, E. Valdinoci, \emph{Hitchhiker's guide to the fractional Sobolev spaces}, Bull. Sci. Math., 136 (2012), 521-573.

\bibitem{FK} M. Felsinger, M. Kassmann, \emph{Local regularity for parabolic nonlocal operators}, Comm. PDE 38 (2013), 1539-1573.

\bibitem{FG}  R. Frank, L. Geisinger, \emph{Refined semiclassical asymptotics for fractional powers of the Laplace operator}, J. Reine Angew. Math., to appear.

\bibitem{G} L. Geisinger, \emph{A short proof of Weyl's law for fractional differential operators}, J. Math. Phys. 55 (2014), 011504.

\bibitem{GH} P. Glowacki, W. Hebisch, \emph{Pointwise estimates for densities of stable semigroups of measures}, Studia Math. 104 (1992), 243-258.

\bibitem{Gr} G. Grubb, \emph{Spectral results for mixed problems and fractional elliptic operators}, preprint arXiv (Jul. 2014).

\bibitem{Gr2} G. Grubb, \emph{Fractional Laplacians on domains, a development of Hörmander's theory of mu-transmission pseudodifferential operators}, Adv. Math., accepted (2014). 

\bibitem{KS} M. Kassmann, R. W. Schwab, \emph{Regularity results for nonlocal parabolic equations}, Riv. Mat. Univ. Parma 5 (2014), 183-212.

\bibitem{P}  S. I. Pohozaev, \emph{On the eigenfunctions of the equation $\Delta u + \lambda f(u) = 0$}, Dokl. Akad. Nauk SSSR 165 (1965), 1408-1411.

\bibitem{RS} X. Ros-Oton, J. Serra, \emph{The Dirichlet problem for the fractional Laplacian: regularity up to the boundary}, J. Math. Pures Appl. 101 (2014), 275-302.

\bibitem{RS-P} X. Ros-Oton, J. Serra, \emph{The Pohozaev identity for the fractional Laplacian}, Arch. Rat. Mech. Anal. 213 (2014), 587-628.

\bibitem{RS-ext} X. Ros-Oton, J. Serra, \emph{The extremal solution for the fractional Laplacian}, Calc. Var, 50 (2014) 723-750. 

\bibitem{RS-L} X. Ros-Oton, J. Serra, \emph{Boundary regularity for fully nonlinear integro-differential equations}, preprint arXiv (Apr. 2014).

\bibitem{RS-F} X. Ros-Oton, J. Serra, \emph{Regularity theory for general stable operators}, forthcoming.

\bibitem{S} G. Samorodnitsky, M. S. Taqqu, \emph{Stable Non-Gaussian Random Processes: Stochastic Models With Infinite Variance}, Chapman and Hall, New York, 1994.

\bibitem{Serra} J. Serra, \emph{Regularity for fully nonlinear nonlocal parabolic equations with rough kernels}, preprint arXiv (Jan. 2014).


\end{thebibliography}
\end{document}